\documentclass[11pt]{amsart}

\usepackage[a4paper,hmargin=3.5cm,vmargin=4cm]{geometry}
\usepackage{amsfonts,amssymb,amscd,amstext,verbatim}

\usepackage{inputenc}
\usepackage{hyperref}
\usepackage{verbatim}

\usepackage{fancyhdr}
\usepackage{pstricks,pstcol,pst-plot}

\definecolor{Black}{cmyk}{0,0,0,1}
\definecolor{OrangeRed}{cmyk}{0,0.6,1,0}            
\definecolor{DarkBlue}{cmyk}{1,1,0,0.20}
\definecolor{myblue}{rgb}{0.66,0.78,1.00}
\definecolor{Violet}{cmyk}{0.79,0.88,0,0}
\definecolor{Lavender}{cmyk}{0,0.48,0,0}

\usepackage{times}
\usepackage{enumerate}
\usepackage{titlesec}
\usepackage{mathrsfs}

\titleformat{\section}
{\filcenter\bfseries\large} {\thesection{.}}{0.2cm}{}
\titleformat{\subsection}[runin]
{\bfseries} {\thesubsection{.}}{0.15cm}{}[.]
\titleformat{\subsubsection}[runin]
{\em}{\thesubsubsection{.}}{0.15cm}{}[.]

\usepackage[up,bf]{caption}

\newtheorem{theorem}{Theorem}[section]
\newtheorem{proposition}[theorem]{Proposition}

\newtheorem{lemma}[theorem]{Lemma}
\newtheorem{corollary}[theorem]{Corollary}

\theoremstyle{definition}
\newtheorem{definition}[theorem]{Definition}
\newtheorem{remark}[theorem]{Remark}

\numberwithin{equation}{section}
\numberwithin{figure}{section}

\newcommand\Acal{\mathcal{A}}
\newcommand\Bcal{\mathcal{B}}

\newcommand\Ecal{\mathcal{E}}

\newcommand\Gcal{\mathcal{G}}
\newcommand\Hcal{\mathcal{H}}

\newcommand\Ocal{\mathcal{O}}

\newcommand\Ascr{\mathscr{A}}

\newcommand\Cscr{\mathscr{C}}

\newcommand\Rscr{\mathscr{R}}

\newcommand\B{\mathbb{B}}

\newcommand\CP{\mathbb{CP}}

\newcommand\N{\mathbb{N}}

\newcommand\R{\mathbb{R}}

\newcommand\Z{\mathbb{Z}}

\newcommand\ggot{\mathfrak{g}}

\newcommand\igot{\mathfrak{i}}

\renewcommand\igot{\mathfrak{i}}

\renewcommand\imath{\igot}

\newcommand\lra{\longrightarrow}

\newcommand\wh{\widehat}
\newcommand\di{\partial}
\newcommand\dibar{\overline\partial}

\newcommand\dist{\mathrm{dist}}

\newcommand\supp{\mathrm{supp}}

\newcommand\Aut{\mathrm{Aut}}

\newcommand\Id{\mathrm{Id}}

\def\dist{\mathrm{dist}}

\newcommand\Ocalc{\overline{\mathcal{O}}}
\newcommand\Ocalcl{\Ocalc_{\mathrm{loc}}}

\begin{document}

\fancyhead[LO]{Mergelyan's and Arakelian's theorems for manifold-valued maps}
\fancyhead[RE]{F.\ Forstneri\v c} 
\fancyhead[RO,LE]{\thepage}

\thispagestyle{empty}

\vspace*{1cm}
\begin{center}
{\bf\LARGE Mergelyan's and Arakelian's theorems for manifold-valued maps}

\vspace*{0.5cm}

{\large\bf  Franc Forstneri\v c} 
\end{center}


\vspace*{1cm}

\begin{quote}
{\small
\noindent {\bf Abstract}\hspace*{0.1cm}
In this paper we show that Mergelyan's theorem holds for maps from open Riemann surfaces 
to Oka manifolds.  This is used to prove the analogue of Arakelian's theorem on uniform approximation 
of holomorphic maps from closed subsets of plane domains to any compact complex homogeneous
manifold. 

\vspace*{0.2cm}

\noindent{\bf Keywords}\hspace*{0.1cm} holomorphic map, homogeneous manifold, elliptic manifold, Oka manifold, Riemann surface, Arakelian's theorem, Mergelyan's theorem

\vspace*{0.1cm}


\noindent{\bf MSC (2010):}\hspace*{0.1cm}}  30E10, 32E10, 32E30, 32H02
\end{quote}


\section{Introduction} 
\label{sec:intro}

The goal of this paper is to extend some results of holomorphic approximation theory
to maps with values in complex manifolds. We focus on two classical approximation
theorems: Mergelyan's theorem \cite{Mergelyan1951}, extended  to compact sets
in open Riemann surfaces by E.\ Bishop  \cite{Bishop1958DMJ}, and 
Arakelian's theorem \cite{Arakelian1964}. 

Given a closed set $E$ in a complex manifold $X$, we denote
by $\Ascr(E)$ the algebra of all continuous functions on $E$ which are holomorphic
in the interior $\mathring E$. In a series of papers beginning in 1964, 
N.\ U.\ Arakelian \cite{Arakelian1964,Arakelian1968,Arakelian1971} 
proved that the following conditions are equivalent for a closed set $E$ in a domain $X\subset \mathbb{C}$: 
\begin{itemize}
\item[\rm (a)]\ Every function in $\Ascr(E)$ is a uniform limit of functions holomorphic on $X$.
\item[\rm (b)]\ The complement $X^*\setminus E$ of $E$ in the one point compactification 
$X^*=X\cup\{*\}$ of $X$ is connected and locally connected.
\end{itemize}
A closed set $E$ satisfying these equivalent conditions is called an {\em Arakelian set} in $X$. 
When $E$ is compact, condition (b) simply says that $X\setminus E$ is connected, and in this case
Arakelian's theorem coincides with S.~N. Mergelyan's theorem \cite{Mergelyan1951}. 

In Sect.\ \ref{sec:proof1} we prove the following version of 
Arakelian's theorem for maps from plane domains into any compact complex homogeneous manifold.
See also Theorem \ref{th:Arakelian2} and Corollary \ref{cor:Arakelian2} for a generalisation to maps 
from more general open Riemann surfaces.

%
%
\begin{theorem} \label{th:Arakelian1}
Assume that $Y$ is a compact complex homogeneous manifold or $Y=\mathbb{C}^n$. 
If $E$ is an Arakelian set in a domain $X\subset \mathbb{C}$, then every continuous map 
$X\to Y$ which is holomorphic in $\mathring E$ can be approximated 
uniformly on $E$ by holomorphic maps $X\to Y$.

If in addition $M$ is a closed subset of $Y$ of Hausdorff dimension $<2\dim Y-2$
then the approximating maps may be chosen to have range in $Y\setminus M$.
\end{theorem}

As an interesting special case, we see that Arakelian's theorem holds for maps from
plane domains into any projective space $\CP^n$.

The distance between maps into $Y$ is measured with respect to a fixed  Riemannian metric on $Y$; 
due to compactness any two such metrics are comparable, and hence the statement of the theorem is 
independent of the choice of the metric. On $\mathbb{C}^n$ we use the standard Euclidean metric.
The main case is when $M=\varnothing$; the result for maps into $Y\setminus M$ follows 
by applying the transversality theorem, moving the map slightly so that the image 
misses $M$ (see Lemma \ref{lem:avoiding}). By topological reasons it is necessary 
to assume that the map $E\to Y$ to be approximated extends to a continuous map $X\to Y$.

Global approximation theorems, such as Theorem \ref{th:Arakelian1}, 
fail in general without assuming a suitable holomorphic flexibility property of the target manifold.
(A discussion of flexibility properties of complex manifolds can be found in 
\cite[Chapter 7]{Forstneric2017E}.) In particular, if $Y$ is Kobayashi hyperbolic
then a nonconstant holomorphic map from the disc into $Y$ 
cannot be approximated by holomorphic maps $\mathbb{C}\to Y$.
It is likely impossible to characterize the class of compact complex manifolds $Y$ for which 
Theorem \ref{th:Arakelian1} holds. Natural candidates are Oka manifolds. 
A complex manifold $Y$ is said to be an {\em Oka manifold} 
if the Runge approximation theorem holds for maps $X\to Y$ from any Stein manifold $X$ 
(or, more generally, from any reduced Stein space), with approximation on compact 
$\Ocal(X)$-convex subsets of $X$. (For $Y=\mathbb{C}$, this is the Oka-Weil theorem.)  
For the precise statement, see \cite[Theorem 5.4.4]{Forstneric2017E}
and its corollaries. Introductory surveys of Oka theory can be found in 
\cite{Forstneric2013KAWA,ForstnericLarusson2011}. Every complex homogeneous manifold
is an Oka manifold according to Grauert \cite{Grauert1958MA} 
(see also \cite[Proposition 5.6.1]{Forstneric2017E}). Although we do not know whether 
Theorem \ref{th:Arakelian1} holds for maps to all Oka manifolds,
we will show that Mergelyan's theorem does.

If $K$ is a compact set in a Riemann surface $X$ then a 
relatively compact connected component of $X\setminus K$ is called a {\em hole} of $K$ (in $X$).
The theorem of S.\ N.\ Mergelyan \cite{Mergelyan1951} from 1951, extended 
to open Riemann surfaces by E.\ Bishop \cite{Bishop1958DMJ} in 1958, says that 
for any compact set $K$ without holes in an open Riemann surface $X$, every function
$f\in\Ascr(K)$ is a uniform limit of functions in $\Ocal(X)$.
With the aid of a theorem of E.\ Poletsky \cite{Poletsky2013} (see Theorem \ref{th:Poletsky3.1}) 
we obtain the following extension of Mergelyan's theorem to maps into Oka manifolds.
This provides the induction step in the proof of Theorem \ref{th:Arakelian1}. 

%
%
\begin{theorem}[Mergelyan theorem for maps from Riemann surfaces to Oka manifolds] 
\label{th:Mergelyan-Oka}
Assume that $K$ is a compact set without holes in an open Riemann surface $X$,  
and let $Y$ be an Oka manifold.
Then, every continuous map $f\colon X\to Y$ which is holomorphic in $\mathring K$
can be approximated uniformly on $K$ by holomorphic maps $X\to Y$ homotopic to $f$.
\end{theorem}

We also have a local Mergelyan approximation theorem for maps into an arbitrary complex
manifold. Let us recall the following notion.

%
%
\begin{definition}\label{def:LMP}
A compact set $K$ in a Riemann surface $X$ has the {\em Mergelyan property}
(also called the {\em Vitushkin property})
if every function in $\Ascr(K)$ can be approximated uniformly on $K$ by functions 
holomorphic in open neighborhoods of $K$ in $X$.
\end{definition}

%
%
\begin{theorem}[Local Mergelyan theorem for manifold-valued maps] 
\label{th:Mergelyan-local}
If $X$ is a Riemann surface and $K$ is a compact set in $X$ with the Mergelyan property,
then $K$ also has the Mergelyan property for maps to an arbitrary complex manifold $Y$:
every continuous map $f\colon K\to Y$ which is holomorphic in $\mathring K$
can be approximated uniformly on $K$ by holomorphic maps $U=U_f\to Y$
from open neighborhoods of $K$ in $X$.
\end{theorem}

Theorems \ref{th:Mergelyan-Oka} and \ref{th:Mergelyan-local} are proved in Sect.\ \ref{sec:preliminaries}. 

In view of Runge's theorem \cite{Runge1885}, 
a compact set $K$ in $\mathbb{C}$ has the Mergelyan property if and only if
$\Ascr(K)$ equals $\Rscr(K)$, the uniform closure in $\Cscr(K)$ of the space of 
rational functions with poles off $K$. A characterization of this class of plane compacts 
in terms of the continuous analytic capacity was given by 
A.\ G.\ Vitushkin in 1966 \cite{Vitushkin1966,Vitushkin1967}. 
See also the exposition in T.\ W.\ Gamelin's book \cite{Gamelin1984}.

Earlier instances of Mergelyan's theorem for manifold-valued maps from compact sets in $\mathbb{C}$
were obtained by K.\ K{\"o}nigsberger (1986) (for maps to complex Lie groups;
his paper is mentioned in \cite{Dietmair1993} and 
\cite{Winkelmann1998}, but is not listed by MathSciNet or Zentralblatt),
T.\ Dietmair \cite{Dietmair1993} (for maps to complex homogeneous manifolds), and 
J.\ Winkelmann \cite{Winkelmann1998} (for maps to $\mathbb{C}^2\setminus \R^2$, and to $\mathbb{C}^n\setminus A$
where $A$ is a closed subset of $\mathbb{C}^n$ of Hausdorff dimension $<2n-2$).
The first two mentioned results are special cases of Theorem \ref{th:Mergelyan-Oka}
since every complex homogeneous manifold is Oka. This is not the case for
Winkelmann's theorem concerning maps to the domain $\mathbb{C}^2\setminus \R^2$ 
which is not known to be Oka. Mergelyan's theorem for $\Cscr^k$ 
functions $(k\in \N=\{1,2,\ldots\})$ on compact sets in $\mathbb{C}$ was proved by 
J.\ Verdera in 1986, \cite{Verdera1986PAMS}, who showed that
for every compact set $K$ in $\mathbb{C}$ and compactly supported function $f\in \Cscr^k(\mathbb{C})$
 such that $\di f/\di \bar z$ vanishes on $K$ to order $k-1$, 
$f$ can be approximated in $\Cscr^k(\mathbb{C})$ by functions holomorphic in neighborhoods of $K$. 
Verdera's result extends to manifold-valued maps by the proof of Theorem \ref{th:Mergelyan-local}.

Some Mergelyan type approximation theorems are also known for maps from higher dimensional
Stein manifolds. For example, if $D$ is a relatively compact strongly 
pseudoconvex domain in a Stein manifold $X$ and $Y$ is an arbitrary complex 
manifold, then any map $\bar D\to Y$ of class $\Ascr^k(D,Y)$ for some $k\in\Z_+=\{0,1,2,\ldots\}$ 
can be approximated in the $\Cscr^k$ topology by maps holomorphic in open 
neighborhoods of $\bar D$ in $X$ (see \cite[Theorem 1.2]{DrinovecForstneric2008FM}
and  \cite[Theorem 8.11.4]{Forstneric2017E}). Approximation of manifold-valued maps on 
Stein compacts of the form $K\cup M$, where $K$ is a 
Stein compact and $M$ is a totally real submanifold, was obtained in 
\cite[Theorem 3.2]{Forstneric2004AIF}.
See also \cite[Corollaries 5.4.6, 5.4.7 and Theorem 8.11.4]{Forstneric2017E} and
the survey \cite{FFW2018} by J.\ E.\ Forn{\ae}ss, E.\ F.\ Wold and the author.

%
%
It is natural to ask whether Arakelian's theorem, and its extensions presented in this paper,
holds for maps from more general open Riemann surfaces.
In 1975, P.\ M.\ Gauthier and W.\ Hengartner \cite{GauthierHengartner1975} proved that 
Arakelian's condition (the complement $X^*\setminus E$ of $E$ in the one point 
compactification $X^*=X\cup\{*\}$ of $X$ is connected and locally connected)
is necessary for uniform approximation of functions on a closed subset $E$ 
in an arbitrary connected open Riemann surface $X$, but the converse fails in general.
Several sufficient conditions can be found in the papers by P.\ M.\ Gauthier \cite{Gauthier1980} 
and A.\ Boivin and P.\ M.\ Gauthier \cite[pp.\ 119--121]{BoivinGauthier2001}.
In Sect.\ \ref{sec:Arakelian2} we give a sufficient condition
on an open Riemann surface $X$ (see Definition \ref{def:CondBH})
which makes it possible to prove the Arakelian theorem for maps from $X$
to compact homogeneous manifolds (see Theorem \ref{th:Arakelian2} and Corollary \ref{cor:Arakelian2}).

It is also natural to ask whether Arakelian's theorem holds for certain closed subsets $E$ 
in pseudoconvex domains $X\subset\mathbb{C}^n$ for $n>1$, or in more general Stein manifolds.
A special case studied in the literature are closed sets of the form $E=K\cup M$, where $K$ is 
a compact $\Ocal(X)$-convex subset of $X$ and $M$ is a (possibly stratified) totally real submanifold of $X$. 
Assuming that $E=K\cup M$ as above 
is $\Ocal(X)$-convex and has the {\em bounded exhaustion hulls property} 
(see Definition \ref{def:BEH2}), Carleman type approximation theorems for functions 
in the fine topology on $E$ were obtained by P.\ Manne \cite{Manne1993}, 
P.\ Manne, E.\ F.\ Wold, and N.\ {\O}vrelid \cite{ManneWoldOvrelid2011},
and B.\ S.\ Magnusson and E.\ F.\ Wold \cite{MagnussonWold2016}. 
Recently, B.\ Chenoweth \cite{Chenoweth2018} 
extended these results to maps with values in an arbitrary Oka manifold. 
On the other hand, nothing seems known about uniform approximation  
on closed sets $E$ whose interior $\mathring E$ fails to be relatively compact. 
The main problem seems to be lack of holomorphic integral kernels
with properties comparable to those of the Cauchy-Green kernel in the plane.
Such kernels have been constructed on strongly pseudoconvex domains by 
G.\ M.\ Henkin \cite{Henkin1969} and E.\ Ram{\'i}rez de Arellano \cite{Arellano1969} 
(see also \cite[Sect.\ 3]{ForstnericLowOvrelid2001} and the monographs \cite{HenkinLeiterer1984,LiebMichel2002,Range1986}). Unlike the Cauchy-Green kernel, 
Henkin-Ram{\'i}rez kernels depend on the domain, and it seems difficult, if not impossible, 
to apply them in Arakelian type approximation. 

We wish to draw reader's attention to the recent survey \cite{FFW2018}
of holomorphic approximation theory, with emphasis on generalisations
of Runge's, Mergelyan's, Carleman's and Arakelian's theorems 
to higher dimensional domains and manifold-valued maps.

%
%

\section{Preliminaries}\label{sec:preliminaries}

Let $X$ be a complex manifold. We denote by $\Cscr(X)$ and $\Ocal(X)$ the Frech{\'e}t algebras of all 
continuous and holomorphic functions on $X$, respectively; these spaces carry
the compact-open topology. Given a compact set $K$ in $X$, we denote by $\Cscr(K)$ the Banach algebra 
of all continuous complex valued functions on $K$ endowed with the sup-norm, 
by $\Ocal(K)$ the algebra of all functions $f$ that are holomorphic in a neighborhood $U_f\subset X$ 
of $K$ (depending on the function) with inductive limit topology, 
and by $\Ocalc(K)$ the uniform closure of $\{f|_K: f\in \Ocal(K)\}$ in $\Cscr(K)$. 
By $\Ascr(K)$ we denote the closed subalgebra of $\Cscr(K)$
consisting of all continuous function $K\to\mathbb{C}$ which are holomorphic in the interior $\mathring K$.
If $r\in \{0,1,2,\ldots, \infty\}$, we let $\Cscr^r(K)$ denote the space of all functions on $K$
which extend to $r$-times continuously differentiable functions on $X$, 
and $\Ascr^r(K)=\Cscr^r(K)\cap \Ocal(\mathring K)$.
Given a complex manifold $Y$, we use the analogous notation $\Ocal(X,Y)$, $\Ocal(K,Y)$,
$\Ascr(K,Y)$, etc., for the corresponding classes of maps into $Y$. Furthermore, we denote by
\[
		\Ocalcl(K,Y)
\]
the set of all maps $f\in \Ascr(K,Y)$ which are locally approximable by holomorphic maps, in the sense 
that every point $x \in K$ has an open neighborhood $U\subset X$ 
such that $f|_{K\cap \overline U}\in \Ocalc(K\cap \overline U)$. 
(Note that uniform approximability on a compact set is independent of the choice
of a Riemannian distance function on $Y$.) We have natural inclusions
\[
	 \{f|_K : f\in \Ocal(K,Y)\} \subset \Ocalc(K,Y) \subset \Ocalcl(K,Y)\subset \Ascr(K,Y).
\]
When $Y=\mathbb{C}$, we delete it from the notation.

We say that the space $\Ascr(K,Y)$ enjoys the {\em Mergelyan property} if
\begin{equation}\label{eq:AO}
	\Ascr(K,Y) = \Ocalc(K,Y),
\end{equation}
and that it enjoys the {\em local Mergelyan property} if 
\begin{equation}\label{eq:LMP}
	\Ocalcl(K,Y) = \Ascr(K,Y).
\end{equation}

According to Bishop's localization theorem \cite{Bishop1958PJM}, the Mergelyan property for functions
on a compact set $K$ in a Riemann surface $X$ is localizable: 

{\em Given $f\in\Cscr(K)$, if every point $x\in K$ has a compact neighborhood $D_x \subset X$ such that
$f|_{K\cap D_x}\in \Ocalc(K\cap D_x)$, then $f\in \Ocalc(K)$.} 

The following converse result is due to A.\ Boivin and B.\ Jiang  \cite[Theorem 1]{BoivinJiang2004}: 

{\em Let $E$ be a closed subset of a Riemann surface $X$. If $\Ascr(E)=\Ocalc(E)$, 
then $\Ascr(E\cap D)=\Ocalc(E\cap D)$ holds for every closed parametric disc $D\subset X$.}

Recall that a closed parametric disc is the preimage $D=\phi^{-1}(\Delta)$ of a closed disc 
$\Delta \subset \phi(U)\subset \mathbb{C}$, where $(U,\phi)$ is a holomorphic chart on $X$.

For compact sets $K$ in $\mathbb{C}$ we have $\Ocalc(K) = \Rscr(K)$ by Runge's theorem,
and hence the Mergelyan property $\Ascr(K) = \Ocalc(K)$
is equivalent to $\Ascr(K) = \Rscr(K)$. This property of $K$ was characterized
by Vitushkin \cite{Vitushkin1966} in terms of continuous analytic capacity.

The following observation amounts to a standard application of the Bishop-Narasimhan-Remmert 
embedding theorem  (see \cite[Theorem 2.4.1]{Forstneric2017E}) and the Docquier-Grauert tubular neighborhood 
theorem (see \cite{DocquierGrauert1960} or \cite[Theorem 3.3.3]{Forstneric2017E}). 

%
%
\begin{lemma}\label{lem:Steinnbd}
Assume that $K$ is a compact set in a complex manifold $X$ satisfying the Mergelyan
property for functions: $\Ocalc(K)  = \Ascr(K)$. Let $Y$ be a complex manifold, and assume that $f\in \Ascr(K,Y)$ 
satisfies one of the following conditions.
\begin{itemize}
\item[\rm (a)] The image $f(K)\subset Y$ has a Stein neighborhood in $Y$. 
\item[\rm (b)] The graph $G_f = \bigl\{(x,f(x)) : x\in K\bigr\}$ has a Stein neighborhood in $X\times Y$.
\end{itemize} 
Then, $f\in \Ocalc(K,Y)$.
\end{lemma}

\begin{proof}
Assume that condition in (b) holds.  We embed a Stein neighborhood of the graph $G_f$  
as a complex submanifold $\Sigma$ of a Euclidean space $\mathbb{C}^N$, apply the hypothesis 
$\Ocalc(K)  = \Ascr(K)$ componentwise, and compose the resulting maps into $\mathbb{C}^N$ 
(which are holomorphic in open neighborhoods of $K$) with a holomorphic 
retraction onto $\Sigma$. This gives  a sequence of approximating holomorphic maps 
from open neighborhoods of $K$ into $Y$. A similar argument applies in case (a).
\end{proof}

Recall that a compact set $K$ in a complex manifold $X$ is said to be a {\em Stein compact}
if $K$ admits a basis of open Stein neighbourhoods in $X$.
The following theorem is due to E.\ Poletsky \cite[Theorem 3.1]{Poletsky2013};
see also \cite[Theorem 32]{FFW2018} and the related discussion.

%
%
\begin{theorem}[Poletsky \cite{Poletsky2013}] \label{th:Poletsky3.1} 
Let $K$ be a Stein compact in a complex manifold $X$, and let $Y$ be an arbitrary complex manifold.
For every $f\in \Ocalcl(K,Y)$ the graph of $f$ on $K$ is a Stein compact in $X\times Y$. 
In particular, if $\Ascr(K,Y)$ has the local Mergelyan property \eqref{eq:LMP}, then 
the graph of every map $f\in \Ascr(K,Y)$ is a Stein compact in $X\times Y$.
\end{theorem}

Poletsky's  proof uses the technique of {\em fusing plurisubharmonic functions}.
It is similar in spirit  to the proofs of Y.-T.\ Siu's theorem  \cite{Siu1976} on the existence
of open Stein neighborhoods of Stein subvarieties, 
given by M.\ Col{\c{t}}oiu \cite{Coltoiu1990} and J.-P.\ Demailly \cite{Demailly1990}.
(See also \cite[Sect.\ 3.2]{Forstneric2017E}.)
In the special case when $K$ is the closure of a strongly pseudoconvex domain in a Stein manifold,
it was proved beforehand, and by a different method, that the graph  of 
any map $f\in \Ascr(K,Y)$ is a Stein compact in $X\times Y$
(see  \cite[Corollary 1.3]{Forstneric2007AJM}). 

%
%
\begin{proof}[Proof of Theorem \ref{th:Mergelyan-local}]
Since every open Riemann surface is a Stein manifold 
(see H.\ Behnke and K.\ Stein \cite{BehnkeStein1949}), 
any compact subset $K$ of an open Riemann surface $X$ is a Stein compact.
Let $f\in\Ascr(K,Y)$. Pick a point $p\in K$ and choose a closed parametric disc $D\subset X$
around $p$. According to the theorem of Boivin and Jiang \cite[Theorem 1]{BoivinJiang2004}
mentioned above, the assumption $\Ascr(K)=\Ocalc(K)$ implies $\Ascr(K\cap D)=\Ocalc(K\cap D)$. 
By choosing $D$ small enough, $f(K\cap D)$ lies in a  Stein domain in $Y$, 
and Lemma \ref{lem:Steinnbd} implies that $f$ is approximable uniformly on $K\cap D$
by maps that are holomorphic in neighborhoods of $K\cap D$. 
This shows that $f\in \Ocalcl(K,Y)$. Theorem \ref{th:Poletsky3.1} 
implies that the graph of $f$ over $K$ has a Stein neighborhood in $X\times Y$, 
and Lemma \ref{lem:Steinnbd} shows that $f\in \Ocalc(K,Y)$.
\end{proof}

%
%
\begin{proof}[Proof of Theorem \ref{th:Mergelyan-Oka}]
Since $K$ has no holes in $X$, Bishop-Mergelyan theorem \cite{Bishop1958DMJ,Mergelyan1951} 
shows that $\Ascr(K)=\Ocalc(K)$. Theorem \ref{th:Mergelyan-local} implies that $\Ascr(K,Y)=\Ocalc(K,Y)$
holds for every complex manifold $Y$. Assume now that $Y$ is an Oka manifold and 
that $f\in \Ascr(K,Y)$ extends to a continuous map $f\colon X\to Y$. Let $g\colon U\to Y$ be a holomorphic 
map in a neighborhood of $K$ approximating $f$ uniformly on $K$. 
By gluing $g$ with $f$ on $U\setminus K$ we obtain a continuous map $\tilde g\colon X\to Y$ 
which agrees with $g$ in a smaller neighborhood of $K$ and is homotopic to $f$. It follows from 
Oka theory (see \cite[Theorem 5.4.4]{Forstneric2017E}) that $\tilde g|_K$ is a uniform limit 
of holomorphic maps $F\colon X\to Y$ homotopic to $\tilde g$, and hence to $f$.
\end{proof}

%
%
\begin{remark}\label{rem:Poletsky}
Poletsky stated the following \cite[Corollary 4.4]{Poletsky2013}:

\smallskip
\noindent (*)\  {\em If  $K$ is a Stein compact in a complex manifold $X$ 
such that $\Ascr(K)$ has the Mergelyan property, then $\Ascr(K,Y)$ has 
the Mergelyan property for any complex manifold $Y$.}
\smallskip

The proof  in \cite{Poletsky2013} tacitly assumes that under assumptions of the corollary
the space $\Ascr(K,Y)$ has the local Mergelyan property, but no explanation for this is given. 
As pointed out in the proof of Theorem \ref{th:Mergelyan-local},  this holds for compact sets in 
Riemann surfaces in view of the theorem of Boivin and Jiang \cite[Theorem 1]{BoivinJiang2004}.
It is easily seen that $\Ascr(K,Y)$ has the local Mergelyan property when the set 
$K$ has $\Cscr^1$ boundary.

We refer the reader to \cite[Sect.\ 7.2]{FFW2018} for a more complete discussion of Mergelyan type
approximation theorems for manifold-valued maps.
\qed\end{remark}

%
%

\section{A splitting lemma on closed Cartan pairs in $\mathbb{C}$}\label{sec:splitting}

In this section we prepare an important technical tool that will be 
used in the proof of Theorem \ref{th:Arakelian1}; see Lemma \ref{lem:splitting}. 
We begin by introducing appropriate geometric configurations 
of closed sets in Riemann surfaces.

%
%
\begin{definition}\label{def:Cpair}
Let $X$ be a Riemann surface. A pair of closed subsets $(A,B)$ of $X$ is a {\em Cartan pair}
if it satisfies the following two conditions.
\begin{itemize}
\item[\rm (a)] The set $K=A\cap B$ is compact. 
\item[\rm (b)]  $\overline{A\setminus B}\, \cap\, \overline {B\setminus A}=\varnothing$.
\end{itemize}
\end{definition}

The sets $A$ and $B$ in the above definition need not be compact.
This is a variation of the usual notion of a Cartan pair (see \cite[Definition 5.7.1]{Forstneric2017E}).

We denote by $T_K$ the Cauchy-Green operator with support on a compact set $K\subset \mathbb{C}$:
\begin{equation}\label{eq:TK}
	T_K(g)(z) = \frac{1}{\pi} \int_K \frac{g(\zeta)}{z-\zeta}\, du\,dv,
	\qquad z\in \mathbb{C},\ \zeta =u+\imath v.
\end{equation}
Recall that for any $g\in L^p(K)$, $p>2$, $T_K(g)$ is a bounded continuous function on $\mathbb{C}$ 
that is holomorphic on $\mathbb{C}\setminus K$, vanishes at infinity, and satisfies the uniform 
H\"older condition with exponent $\alpha=1-2/p$; furthermore, 
$T_K\colon L^p(K)\to \Cscr^\alpha(\mathbb{C})$ is a continuous
linear operator.  (See L.\ Ahlfors \cite[Lemma 1, p.\ 51]{Ahlfors2006} or A.\ Boivin and P.\ Gauthier
\cite[Lemma 1.5]{BoivinGauthier2001}.)  In particular, $T_K$ maps $\Cscr(K)$ boundedly into 
the space $\Cscr_b(\mathbb{C})$ of bounded continuous functions on $\mathbb{C}$ endowed with the 
supremum norm. Moreover, $\dibar \,T_K(g)=g$ holds in the sense of distributions, 
and in the classical sense on any open subset $\Omega\subset \mathbb{C}$ where $g$ is of class $\Cscr^1$.
(For the precise Ahlfors-Beurling estimate of $T_K$  in terms of the  area of $K$,
we recommend the paper by T.\ W.\ Gamelin and D.\ Khavinson \cite{GamelinKhavinson1989}.
Another excellent source for this topic is the book \cite{AstalaIwaniecMartin2009}
by K.\ Astala, T.\ Iwaniec and G.\ Martin; see in particular Section 4.3.) 

The following lemma provides a solution of the Cousin-I problem with bounds on 
a Cartan pair in a plane domain.

%
%
%
%
\begin{lemma}\label{lem:linear-splitting}
Let $(A,B)$ be a Cartan pair in a domain $X\subset \mathbb{C}$ and set $K=A\cap B$. 
There exist bounded linear operators 
$\Acal \colon \Ascr(K) \to \Ascr(A)$, $\Bcal \colon  \Ascr(K) \to \Ascr(B)$, satisfying
\begin{equation}\label{eq:left-right}
        g = \Acal(g) - \Bcal (g) \quad {\rm for\ every}\ \ g\in \Ascr(K).
\end{equation}
\end{lemma}

\begin{proof}
By condition (b) in Definition \ref{def:Cpair} there is a smooth function 
$\chi \colon X \to [0,1]$ such that $\chi=0$ in a neighborhood of $\overline {A\setminus B}$ 
and $\chi=1$ in a neighborhood of $\overline{B\setminus A}$.
For any $g\in \Ascr(K)$, the product $\chi g$ extends to a continuous function on
$A$ that vanishes on $\overline {A\setminus B}$, and $(\chi-1)g$ extends to a continuous function on
$B$ that vanishes on $\overline {B\setminus A}$. Furthermore, 
\[
    	\dibar(\chi g)= \dibar((\chi-1)g)=g\, \dibar \chi
\]
is a continuous function on $K$ which is smooth on $\mathring K$
(note that $\dibar \chi(z)=\di\chi/\di\bar z$ is smooth). 
Since $\dibar \chi$ vanishes on $\overline{A\setminus B}\, \cup\, \overline {A\setminus B}$,
we see that $g\, \dibar \chi$ extends to a continuous function on 
$A\cup B$ which is smooth in the interior and supported in $K$. Set
\[
    \Acal(g) =\chi g - T_K(g\dibar \chi),\qquad  \Bcal(g) =(\chi-1)g - T_K(g\dibar \chi).
\]    
By the properties of the Cauchy-Green operator $T_K$ \eqref{eq:TK},
$\Acal$ and $\Bcal$ are sup-norm bounded linear operators $\Ascr(K)\to\Cscr(\mathbb{C})$, 
$\Acal(g)|_A\in \Ascr(A)$, $\Bcal(g)|_B \in \Ascr(B)$, and \eqref{eq:left-right} holds.
\end{proof}


Given a compact subset $K$ of $\mathbb{C}$ and an open set $W\subset \mathbb{C}^n$, we consider maps
$\gamma \colon K\times W\to K\times \mathbb{C}^n$ of the form
\begin{equation}\label{eq:gamma}
     \gamma(z,w)=\bigl(z,g(z,w)\bigl),\quad z\in K, \ w\in W,
\end{equation}
where $g=(g_1,\ldots,g_n)\colon K\times W\to \mathbb{C}^n$.
We say that $\gamma$ is of class  $\Ascr(K\times W)$ if $g$ is continuous
on $K\times W$ and holomorphic in $\mathring K \times W$.
Let $\Id(z,w)=(z,w)$ denote the identity map on $\mathbb{C}\times \mathbb{C}^n$.
With $\gamma$ as in \eqref{eq:gamma} we set
\[
        \dist_{K \times W}(\gamma,\Id) = \sup\bigl\{ |g(z,w)-w| : z\in K, \ w\in W\bigr\}.
\]

%
%
%
%
\begin{lemma}	\label{lem:splitting}
Let $(A,B)$ be a Cartan pair in a domain $X\subset \mathbb{C}$ (see Def.\ \ref{def:Cpair}).
Set  $K=A\cap B$. Given a bounded open convex set $0\in W\subset\mathbb{C}^n$ and a
number $r\in(0,1)$, there is a $\delta>0$ satisfying the following property.

For every map $\gamma \colon K \times W \to K \times \mathbb{C}^n$
of the form (\ref{eq:gamma}) and of class $\Ascr(K \times W)$, with
$\dist_{K \times W}(\gamma,\Id) <\delta$, there exist maps
\[
 	 \alpha_\gamma : A \times r W \to A \times \mathbb{C}^n, 
	 \qquad \beta_\gamma :  B \times r W \to B \times \mathbb{C}^n, 
\]
of the form (\ref{eq:gamma}) and of class $\Ascr(A\times rW)$ and $\Ascr(B\times rW)$,
respectively, depending smoothly on $\gamma$, such that $\alpha_{\Id}=\Id$, $\beta_{\Id}=\Id$, and
\[
   \gamma\circ\alpha_\gamma = \beta_\gamma \quad {\rm holds\ on\ } K\times rW.
\]
\end{lemma}

\begin{proof}
This is a version of  \cite[Proposition 5.8.1]{Forstneric2017E} where the reader
can find references to the original sources. Its proof is obtained by  
following the proof of the cited result and using Lemma \ref{lem:linear-splitting}.
For the application of the latter lemma, note that the linear operator $T_K$
\eqref{eq:TK}, and hence also the operators $\Acal ,\Bcal$ in Lemma \ref{lem:linear-splitting},
can also be applied to functions $g(z,w)$ as in \eqref{eq:gamma} depending
holomorphically on a parameter $w\in W\subset \mathbb{C}^n$. Indeed, the function
$T_K(\chi g(\cdotp,w)) \in \Cscr(\mathbb{C})$ depends holomorphically on $w$ if $g$ does
since the cut-off function $\chi$ is independent of the $w$ variable.
It follows that $\Acal(g)$ and $\Bcal(g)$ in \eqref{eq:left-right} 
also depend holomorphically on $w$.
This gives the exact analogue of \cite[Lemma 5.8.2]{Forstneric2017E}
on any Cartan pair $(A,B)$ in a plane domain $X\subset\mathbb{C}$.
By using this result, Lemma \ref{lem:splitting} is obtained by following the proof of 
\cite[Proposition 5.8.1]{Forstneric2017E} step by step.
\end{proof}

The following lemma provides a continuous gluing of an approximating map
on a closed subset with the given globally defined map.

%
%
\begin{lemma}\label{lem:cont-gluing}
Given a compact $\Cscr^1$ manifold $Y$ with a Riemannian distance function $\dist_Y$, 
there is a number $r>0$ such that the following holds. If $E$ is a closed subset of a 
manifold $X$, $f\colon X\to Y$ is a continuous map, and $g\colon E\to Y$ 
is a continuous map satisfying $\sup_{x\in E}\dist_Y(g(x),f(x))<r$, then there is a continuous map 
$\tilde g\colon X\to Y$ that agrees with $g$ on $E$ and is homotopic to $f$.
\end{lemma}

\begin{proof}
We embed $Y$ as a $\Cscr^1$ submanifold of a Euclidean space $\R^N$.
There is an open tubular neighborhood $\Omega\subset \R^N$ of $Y$ 
and a retraction $\rho\colon \Omega\to Y$. Since $Y$ is compact, there is a number $r>0$ such that 
for every point $y\in Y$ the Euclidean ball $\B^N(y,r)\subset \R^N$ is contained in $\Omega$. 
Since any two distance functions on a compact manifold are comparable, 
it suffices to prove the lemma with $\dist_Y$ being the restriction of the  Euclidean distance function 
on $\R^N$ to $Y$ and the number $r$ defined above.
We consider $f$ and $g$ as maps to $\R^N$ with range contained in $Y$.
Assume that $\sup_{x\in E} |f(x)-g(x)|<r$. By Tietze's theorem, we can extend $g$ to a 
continuous map $g_1\colon X\to \R^N$. By continuity there is an open neighborhood $U\subset X$ 
of $E$ such that $\sup_{x\in U}|f(x)-g_1(x)|<r$. Let $\chi\colon X\to [0,1]$ be a continuous
function which equals $1$ on $E$ and satisfies $\supp(\chi)\subset U$.
The map $h=\chi g_1+(1-\chi)f:X\to \R^N$ then agrees with $g$ on $E$, with $f$ on $X\setminus U$,
and it satisfies $\sup_{x\in X} |f(x)-h(x)|<r$. Furthermore, $h_t=th+(1-t)f\colon X\to \R^N$ 
is a homotopy from $h=h_1$ to $f=h_0$ such that  $\sup_{x\in X} |f(x)-h_t(x)|<r$ for all for $t\in[0,1]$.
In particular, the homotopy $h_t$ has range in $\Omega$. The map $\tilde g=\rho\circ h\colon X\to Y$ 
and the homotopy $\tilde g_t=\rho\circ h_t\colon X\to Y$ for $t\in [0,1]$ then satisfy 
the conclusion of the lemma.  
\end{proof}

%
%

\section{Proof of Theorem \ref{th:Arakelian1}}\label{sec:proof1}

We begin by recalling a more convenient interpretation of Arakelian's condition on a closed
set $E$ in a domain $X\subset \mathbb{C}$.
For simplicity of exposition we shall consider the case $X=\mathbb{C}$, noting that the arguments
adapt easily to an arbitrary domain in $\mathbb{C}$. 

Given a closed set $E\subset \mathbb{C}$, we denote by $H_E$ the union of all its holes. 
(Recall that a hole of $E$ is a relatively compact connected component of its complement $\mathbb{C}\setminus E$.)

%
%
\begin{definition} [Bounded exhaustion hulls property] \label{def:BEH}
A closed set $E$ in $\mathbb{C}$  has the {\em bounded exhaustion hulls property} (BEH) 
if the set $H_{E\cup \Delta}$ is bounded for every closed disc $\Delta$ in $\mathbb{C}$. 
\end{definition}

It is well known and easily seen that a closed subset $E\subset \mathbb{C}$ with connected 
complement enjoys the BEH property if and only if $\CP^1\setminus E$ is 
locally connected at $\infty=\CP^1\setminus \mathbb{C}$, i.e., $E$ is a Arakelian set.

%
%
\begin{proof}[Proof of Theorem \ref{th:Arakelian1}]
We follow the scheme of proof of Arakelian's theorem given by 
Rosay and Rudin \cite{RosayRudin1989} (1989), adapting it to manifold-valued maps. 

In the special case when $Y=\mathbb{C}^m$ the theorem holds by applying the original Arakelian's theorem
componentwise, using also Lemma \ref{lem:avoiding} below to obtain a map whose image
avoids a given set  $M\subset \mathbb{C}^m$ as in the statement of the theorem.

From now on we assume that $Y$ is a compact homogeneous manifold. 

Since the set $E\subset\mathbb{C}$ has connected complement and enjoys 
the BEH property, we can inductively find a sequence of closed discs
$\Delta_1\subset \Delta_2\subset \cdots \subset \bigcup_{i=1}^\infty \Delta_i=\mathbb{C}$ such that, letting 
$H_i=H_{E\cup \Delta_i}$ denote the union of holes of $E\cup \Delta_i$, we have that
$\Delta_i\cup \overline H_i \subset \mathring \Delta_{i+1}$ for $i=1,2,\ldots$. Set 
\begin{equation}\label{eq:Ei}
	E_0=E, \qquad  E_i=E\cup\Delta_i\cup H_i\ \  \text{for} \ \ i=1,2,\ldots.
\end{equation}
Note that $E_i$ is a closed Arakelian set in $\mathbb{C}$ and we have that
\[
	E_i\subset E_{i+1},\qquad  E\setminus \Delta_{i+1}=E_i\setminus \Delta_{i+1},
	\qquad  \bigcup_{i=0}^\infty E_i=\mathbb{C}. 
\]
The situation is illustrated in Fig.\ \ref{fig:Ei}, where the disc $\Delta$ will be chosen in \eqref{eq:Delta}.

%
%

\begin{figure}[ht]
\psset{unit=0.5cm} 

\begin{pspicture}(-7,-6)(7,6)

\pscircle[linecolor=Violet](0,0){6}     
\rput(4.8,4.8){$\Delta_{i+1}$}	

\pscircle[linecolor=Violet](0,0){5}     
\rput(2.9,3.3){$\Delta$}			

\pscircle[linecolor=OrangeRed,fillstyle=crosshatch,hatchcolor=yellow](0,0){3.5}  
\rput(1.8,2.1){$\Delta_i$}						                 	

\pscircle[linecolor=DarkBlue,fillstyle=crosshatch,hatchcolor=myblue](0,0){2}    
\rput(0.2,0.5){$\Delta_{i-1}$}					         	          

\pscurve[linecolor=red,linewidth=1.2pt](-6.8,5)(-2,2.5)(-4,1.8)(-1.2,0.5)(-2.7,0.2)(-1.2,-0.3)(-0.5,-2.5)(0,-1)(1.5,-2.3)
(1,-0.5)(4,2)(3,0)(4.2,-0.7)(2.7,-1.5)(3.2,-1.8)(4.5,-3)(2.5,-3.6)(6.8,-5)    

\rput(-4,0.4){$H_i$} \psline[linewidth=0.3pt]{->}(-4,0.8)(-3.4,1.8)
\rput(4.2,0.4){$H_i$}  \psline[linewidth=0.3pt]{->}(4.2,0.8)(3.6,1.7)
				  \psline[linewidth=0.3pt]{->}(4.2,0.1)(3.7,-0.8)

\rput(-2.2,-1.4){$H_{i-1}$} \psline[linewidth=0.3pt]{->}(-2.2,-1.1)(-2.2,0.2)
					\psline[linewidth=0.3pt]{->}(-2,-1.8)(-0.5,-2.2)

\rput(-6,4.2){$E$}\rput(6,-4.2){$E$}

\end{pspicture}
\caption{Sets in the inductive step} 
\label{fig:Ei}
\end{figure}
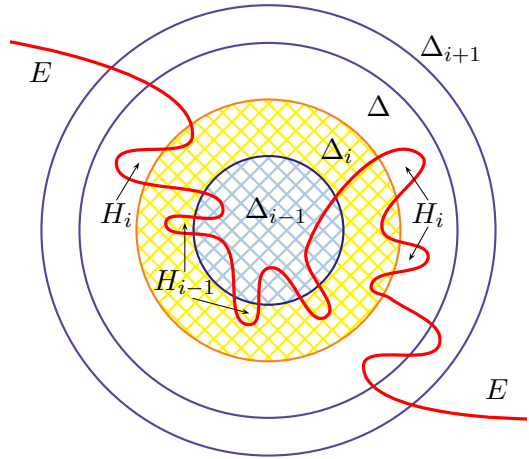
%
%

Let $f=f_0\colon \mathbb{C}\to Y$ be a continuous map such that $f_0|_E \in \Ascr(E)$.
Fix a Riemannian distance function $\dist_Y$ on $Y$. 
(Since $Y$ is compact, any  two such metrics are comparable to each other, so a particular
choice will not matter.) We must show that for every $\epsilon>0$ there is a 
holomorphic map $F\colon \mathbb{C}\to Y$ such that
\begin{equation}\label{eq:condF}
	\sup_{z\in E} \dist_Y(f(z),F(z)) < \epsilon.
\end{equation}
We assume that $0<\epsilon <r$, where $r>0$ is the number for which 
Lemma \ref{lem:cont-gluing} holds.

We shall inductively construct a sequence of continuous maps $f_i\colon\mathbb{C}\to Y$
$(i=1,2,\ldots)$ such that $f_i|_{E_i}\in \Ascr(E_i,Y)$ and the following estimates hold:
\begin{equation}\label{eq:inductionest}
	\dist_Y(f_i(z),f_{i-1}(z)) < 2^{-i}\epsilon, \qquad z\in E_{i-1}.
\end{equation}
Since the sets $E_i$ exhaust $\mathbb{C}$, the sequence $f_i$ converges uniformly on 
compacts in $\mathbb{C}$
to a holomorphic map $F=\lim_{i\to \infty}f_i\colon \mathbb{C}\to Y$ satisfying \eqref{eq:condF}. 

We begin with technical preparations. We shall be using 
the method of gluing holomorphic dominating fibre-sprays, suitably adapted to the case 
at hand. For general background and further details of this technique we refer 
the reader to \cite[Chapter 5]{Forstneric2017E}.

Recall that the manifold $Y$ is assumed to be complex homogeneous, i.e., $Y=\Gcal/\Hcal$ 
where $\Gcal$ is a complex Lie group and $\Hcal$ is a complex Lie subgroup of $\Gcal$. 
Let $\ggot\cong \mathbb{C}^n$ $(n=\dim \Gcal)$  denote the Lie algebra of $\Gcal$, and 
let $\exp\colon \ggot\to\Gcal$ be the associated exponential map. Note that the holomorphic map 
$s\colon Y\times \ggot  \to Y$ defined by
\begin{equation}\label{eq:spray}
	s(y,t) = \exp t \,\cdotp y, \qquad y\in Y,\ t\in \ggot,
\end{equation}
satisfies $s(y,0)=y$ and its partial differential
\begin{equation}\label{eq:domination}
	\frac{\di s(y,t)}{\di t}\bigg|_{t=0} : \ggot=\mathbb{C}^n \longrightarrow T_y Y
\end{equation}
is surjective for every point $y\in Y$.
Thus, $s$ is a {\em dominating spray} on $Y$ defined on the trivial bundle 
$Y\times \ggot\cong Y\times \mathbb{C}^n$.
(See Gromov \cite[4.6B]{Gromov1989} or \cite[p.\ 231]{Forstneric2017E} for the notion of 
dominating sprays and elliptic manifolds. These notions are briefly
recalled in Sect.\ \ref{sec:remark}.) 
By compactness of $Y$ there are positive constants $c_0>0,\ c_1>0$ such that
\begin{equation}\label{eq:c1}
	\dist_Y(y,s(y,t)) \le c_1 |t|,\qquad y\in Y,\ t \in \mathbb{C}^n, \ |t|\le c_0.
\end{equation}
Set $Y':= \mathbb{C}\times Y$ and denote its elements by $y'=(z,y)$, with $z\in\mathbb{C}$ and $y\in Y$. 
The projection 
\[
	\pi : Y'\lra \mathbb{C},\qquad \pi(z,y)=z,
\] 
is a trivial fibre bundle over $\mathbb{C}$ with fibre $Y$.
Let $\Ecal=Y'\times \mathbb{C}^n$ denote the trivial rank $n$ complex 
vector bundle over $Y'$. We identify $Y'$ with the zero section $Y'\times \{0\}^n$ 
of $\Ecal$. Given a point $y'\in Y'$ we denote by $0_{y'}\in \Ecal$ the zero element of
the fibre $\Ecal_{y'}\cong \mathbb{C}^n$. Note that for any vector bundle $\Ecal\to Y'$,
the restriction $T\Ecal|_{Y'}$ of the tangent bundle of the total space 
to the zero section $Y'$ is isomorphic to the direct sum $TY'\oplus \Ecal$.

The spray map $s$ \eqref{eq:spray} defines the {\em dominating fibre-spray} 
$\sigma\colon \Ecal\to Y'$, associated to the trivial fibration $\pi\colon Y'\to \mathbb{C}$, by 
\begin{equation}\label{eq:sigma}
	\sigma(z,y,t)=(z,s(y,t)),\qquad z\in \mathbb{C},\ y\in Y,\ t\in\mathbb{C}^n.
\end{equation}
The domination property of $\sigma$ refers to the fact that for any point $y'=(z,y)\in Y'$,
the restriction of the differential 
\begin{equation}\label{eq:difsigma}
	d\sigma_{0_{y'}}: T_{0_{y'}} \Ecal = T_{y'} Y' \oplus \Ecal_{y'} 
	\longrightarrow T_{y'} Y'
\end{equation}
to the fibre $\Ecal_{y'}=\mathbb{C}^n$ maps $\mathbb{C}^n$ surjectively 
onto the tangent space at $y'=(z,y)$ to the fibre $\pi^{-1}(z)\cong Y$ 
of the projection $\pi\colon Y'\to\mathbb{C}$ (this follows from the domination property
\eqref{eq:domination} of the spray $s$).
This restriction is called the {\em vertical derivative} of the fibre-spray $\sigma$.
For further details on dominating fibre-sprays we refer to \cite[Sect.\ 6.1]{Forstneric2017E}.

We are now ready to explain the induction step $i-1\to i$. 

Assume that $f_{i-1}\colon \mathbb{C}\to Y$ is a continuous map such that $f_{i-1}|_{E_{i-1}} \in\Ascr(E_{i-1})$.
Recall that $\Delta_i\cup \overline H_i  \subset \mathring\Delta_{i+1}$.
Pick a closed disc $\Delta$ in $\mathbb{C}$ such that 
\begin{equation}\label{eq:Delta}
	\Delta_i\cup \overline H_i\subset \Delta \subset \mathring\Delta_{i+1}.
\end{equation}
(See Fig.\ \ref{fig:Ei}.)
Since $E_i=E\cup\Delta_i\cup H_i$ (see \eqref{eq:Ei}), it follows from \eqref{eq:Delta} that 
\begin{equation}\label{eq:EminusDelta}
	E_i\setminus \Delta = E_{i-1}\setminus \Delta = E\setminus \Delta.  
\end{equation}
Since $E_{i-1}$ has no holes,  $E_{i-1} \cap \Delta_{i+1}$ has no holes either. 
As $Y$ is complex homogeneous and hence an Oka manifold, 
Theorem \ref{th:Mergelyan-Oka} (Mergelyan's theorem for maps to Oka manifolds) 
furnishes for any $c>0$ a holomorphic map $h\colon\mathbb{C}\to Y$ satisfying  
\begin{equation}\label{eq:approx-indstep}
	\dist_Y (f_{i-1}(z),h(z))  < c, \qquad z\in E_{i-1}\cap \Delta_{i+1}. 
\end{equation}
The precise value of the constant $c$ will be determined later.  

Consider the compact set
\begin{equation}\label{eq:K} 
	K=E_{i-1}\cap \overline{\Delta_{i+1}\setminus \Delta} 
	= E \cap \overline{\Delta_{i+1}\setminus \Delta},
\end{equation}
where the last identity follows from \eqref{eq:EminusDelta}.
(See Fig.\ \ref{fig:Ei}.) Since $E$ has no holes, $K$ has no holes except perhaps $\Delta$
(this happens for example if $E$ contains the disc $\Delta_{i+1}$).
Therefore, $K$ enjoys the local Mergelyan property 
(see \eqref{eq:LMP}). By Theorem \ref{th:Poletsky3.1} it follows that the graph 
\begin{equation}\label{eq:Gamma}
	\Gamma = \{(z,f_{i-1}(z)) : z\in K\}  \subset \mathbb{C}\times Y
\end{equation}
of  the map $f_{i-1}$ on $K$ admits an open Stein neighborhood $\Omega \subset Y'=\mathbb{C}\times Y$.

Consider the subset $\Ecal'$ of the trivial bundle $\Ecal=Y'\times \mathbb{C}^n$ defined by 
\begin{equation}\label{eq:Eprime}
	\Ecal' =  \left\{(z,y,t): z\in \mathbb{C},\ y\in Y,\ t\in \mathbb{C}^n,\ \frac{\di s(y,t)}{\di t}\bigg|_{t=0}=0\right\}.
\end{equation}
(We have identified $\mathbb{C}^n$ with its tangent space $T_0\mathbb{C}^n$ at the origin, 
considered as a vector subspace of $T_{0_y} (Y\times \mathbb{C}^n)$.)
Since the partial differential $\di s(y,t)/\di t|_{t=0}:\mathbb{C}^n\to T_y Y$ is surjective for every $y\in Y$
(see \eqref{eq:domination}), $\Ecal'$ is a holomorphic vector subbundle of $\Ecal$.
Since the domain $\Omega\subset Y'$ is Stein, Cartan's Theorem B shows that 
\begin{equation}\label{eq:directsum}
	\Ecal|_\Omega = \Ecal'|_\Omega \oplus \Ecal''
\end{equation}
for some holomorphic vector subbundle $\Ecal''$ of $\Ecal|_\Omega = \Omega\times \mathbb{C}^n$ 
(see \cite[Corollary 2.6.6]{Forstneric2017E}). It follows from \eqref{eq:Eprime} and \eqref{eq:directsum}
that for each point $y'=(z,y)\in\Omega$ the restriction of the differential $d\sigma_{0_{y'}}$ \eqref{eq:difsigma}
to the fibre $\Ecal''_{y'}$ maps $\Ecal''_{y'}$ isomorphically onto $T_y Y$.
We claim that, as a consequence of this and the implicit function theorem, 
there exist a smaller neighborhood $\Omega'\subset \Omega$
of $\Gamma$ in $\mathbb{C}\times Y$, a neighborhood $0\in W\subset \mathbb{C}^n$, and a holomorphic map 
\[
	G : Z= \bigl\{(z,y_1,y_2,t): (z,y_1), (z,y_2)\in \Omega',\ t\in W\bigr\} \lra \mathbb{C}^n
\]
satisfying
\begin{equation}\label{eq:Gtrans}
	\sigma(z,y_1,t) = \sigma\bigl(z,y_2,G(z,y_1,y_2,t)\bigr),\qquad (z,y_1,y_2,t)\in Z,
\end{equation}
and
\begin{equation}\label{eq:Gid}
	G(z,y,y,t)  = t,\qquad (z,y)\in \Omega',\ \ t\in W.
\end{equation}
Let us explain the construction of $G$. 
Given $t\in\mathbb{C}^n$ and a point $y'=(z,y)\in \Omega$ we let
\[
	 t = t'_{y'}\oplus t''_{y'} \in \Ecal'_{y'} \oplus \Ecal''_{y'}
\]
be the decomposition corresponding to the holomorphic direct sum \eqref{eq:directsum}. 
Choosing the neighborhoods $\Omega'$ and $W$ as above and small enough, 
the implicit function theorem furnishes a unique map $G$  of the form
\[
	G(z,y_1,y_2,t) = t'_{(z,y_2)} \oplus G''(z,y_1,y_2,t)  \in 
	\Ecal'_{(z,y_2)} \oplus \Ecal''_{(z,y_2)} = \mathbb{C}^n
\]
satisfying conditions \eqref{eq:Gtrans} and \eqref{eq:Gid}.

Recall that $h\colon\mathbb{C}\to Y$ is a holomorphic map satisfying condition \eqref{eq:approx-indstep}.
Choosing the constant $c>0$ in \eqref{eq:approx-indstep} small enough and inserting the values
$y_1=f_{i-1}(z)$ and $y_2=h(z)$  $(z\in K)$ into the map $G$, we obtain the map
\begin{equation}\label{eq:map-g}
	g(z,t) := G(z,f_{i-1}(z),h(z),t)\in\mathbb{C}^n, \qquad z\in K,\ t\in W,
\end{equation}
of class $\Ascr(K\times W)$ which, in view of \eqref{eq:Gtrans}, satisfies the condition
\begin{equation}\label{eq:main1}
	\sigma(z,f_{i-1}(z),t) = \sigma(z,h(z),g(z,t)),\qquad z\in K,\ t\in W.
\end{equation}

Consider the Cartan pair decomposition $(A,B)$ of $E_i=A\cup B$ defined by  
\[
	A=  \overline{E_i\setminus \Delta} = E\setminus \mathring \Delta,
	\qquad 
	B = E_i\cap \Delta_{i+1}
\]
(See Definition \ref{def:Cpair} and Figure \ref{fig:Ei}.)  
From \eqref{eq:EminusDelta} and \eqref{eq:K} we see that
\[
	A\cap B = E_i \cap (\Delta_{i+1} \setminus \mathring \Delta) = 
	E \cap (\Delta_{i+1} \setminus \mathring \Delta)  = K.
\]
Pick a number $0<r_0<1$. Assuming as we may that the holomorphic map
$h\colon\mathbb{C}\to Y$ is sufficiently uniformly close to $f_{i-1}$ on $K$ (see \eqref{eq:approx-indstep}) 
and $g$ is given by \eqref{eq:map-g}, the map $\gamma\colon K\times W\to K\times \mathbb{C}^n$ 
of class $\Ascr(K\times W)$, defined by
\begin{equation}\label{eq:main2}
	\gamma(z,t)= (z,g(z,t)),  \qquad z\in K,\ t\in W,
\end{equation}
is close to the identity map on $K\times W$ in view of \eqref{eq:Gid}, with $\dist_Y(\gamma,\Id)$
depending on the constant $c>0$ in \eqref{eq:approx-indstep}.
By choosing $c$ small enough, Lemma \ref{lem:splitting} furnishes maps
\begin{equation}\label{eq:main3}
\begin{aligned}
	\alpha(z,t) &= (z,a(z,t)),\qquad z\in A,\ t\in r_0 W, \\
	\beta(z,t)   &= (z,b(z,t)), \qquad z\in B,\ t\in r_0 W,
 \end{aligned}
\end{equation}
of class $\Ascr(A\times r_0W)$ and $\Ascr(B\times r_0 W)$, respectively, uniformly close to the identity 
on their respective domains (depending on the constant $c>0$ in \eqref{eq:approx-indstep}) and satisfying 
\begin{equation}\label{eq:main4}
	\gamma\circ \alpha=\beta\quad \text{on}\ \ K\times r_0 W.
\end{equation}
From \eqref{eq:sigma}, \eqref{eq:main1}, \eqref{eq:main2}, \eqref{eq:main3}, 
and \eqref{eq:main4} it follows that
\[
	s\bigl(f_{i-1}(z),a(z,t)\bigr) = s\bigl(h(z),b(z,t)\bigr) \in Y, \qquad z\in K,\ t\in r_0 W.
\]
The two sides of the above equation define a map $f_i\colon E_i\to Y$ of class $\Ascr(E_i)$
given by
\[
	f_i(z) = 
	\begin{cases}
		s\bigl(f_{i-1}(z),a(z,0)\bigr)   & \text{if $z\in A$}, \\
		 s\bigl(h(z),b(z,0)\bigr)	    & \text{if $z\in B$}.
	\end{cases}
\] 
From \eqref{eq:c1} and the construction of $f_i$ we see that the estimate \eqref{eq:inductionest} 
holds provided the constant $c>0$ in \eqref{eq:approx-indstep} is chosen small enough.

Finally, applying Lemma \ref{lem:cont-gluing} we may extend $f_i$ from $E_i$ to a continuous
map $f_i\colon \mathbb{C}\to Y$. This concludes the induction step and proves the theorem for maps to any compact complex homogenous manifold $Y$.
For the last part, we need the following lemma due to
J.\ Winkelmann  \cite[Proposition 2.1]{Winkelmann1998}.

%
%
\begin{lemma}\label{lem:avoiding}
Assume that $Y=\mathbb{C}^m$, or that $Y$ is a compact homogenous manifold of dimension $m$.
Let $M$ be a closed subset of $Y$ of Hausdorff dimension $<2m-2k$ for some $k=1,\ldots,m-1$.
Then for any complex manifold $X$ of dimension $k$, every holomorphic map $X\to Y$ 
can be uniformly approximated by holomorphic maps $X\to Y\setminus M$.
\end{lemma}

\begin{proof}
Let $s\colon \Ecal=Y\times \mathbb{C}^n \to Y$ be a dominating spray of the form \eqref{eq:spray} on $Y$
for some $n\ge m$. (If $Y=\mathbb{C}^m$, we may use the spray 
$\Ecal = \mathbb{C}^m\times \mathbb{C}^m\to\mathbb{C}^m$, $s(z,t)=z+t$.)
Given a holomorphic map $f\colon X\to Y$, we consider the map $\tilde f\colon X\times \mathbb{C}^n\to Y$ 
defined by 
\[
	\tilde f(x,t)= s(f(x),t)\in Y,\qquad x\in X,\ t\in \mathbb{C}^n.
\]
By compactness of $Y$ and the domination property of the spray $s$
there is a ball $W\subset \mathbb{C}^n$ around the origin
such that the partial differential of $\tilde f$ with respect to the variable $t$ is surjective at all
points $(x,t)\in X\times W$. It follows that for a generic $t\in W$
the map $f_t=\tilde f(\cdotp,t)\colon X\to Y$ misses $M$ by dimension reasons 
(see R.\ Abraham \cite{Abraham1963}). By choosing $t$ close to $0$, $f_t$ 
approximates $f$ uniformly on $X$.  This proves the lemma.
\end{proof}

\vspace{-2mm}
This completes the proof of Theorem \ref{th:Arakelian1}.
\end{proof}

%
%
\section{Arakelian's theorem on more general open Riemann surfaces}\label{sec:Arakelian2}

We now inspect and conceptualise the above proof 
to see under which conditions on an open Riemann surface $X$ and
a closed subset $E\subset X$ does Theorem \ref{th:Arakelian1} hold.

Given a compact set $K$ in a complex manifold $X$, its $\Ocal(X)$-convex hull is 
\[
	\wh K=\{z\in X: |f(z)|\le \sup_K |f| \ \text{for all}\ f\in \Ocal(X)\}.
\]
The set $K$ is said to be $\Ocal(X)$-convex if $K=\wh K$.
If $X$ is an open Riemann surface then $\wh K$ is the union of $K$ with all its holes
(see H.\ Behnke and K.\ Stein \cite{BehnkeStein1949}).
Similarly, one defines the $\Ocal(X)$-convex hull of a closed 
subset $E\subset X$ by $\wh E= \bigcup_i \wh{E_i}$ for any
normal exhaustion $E_1\subset E_2\subset \cdots \subset \bigcup_{i=1}^\infty E_i=E$
by compact subsets; the result is independent of the choice of exhaustion.
(See \cite[Definition 2.1]{ManneWoldOvrelid2011} or \cite[Sect.\ 2.1]{Chenoweth2018}.)
We also define 
\begin{equation}\label{eq:hE}
	h(E) = \overline {\wh{E} \setminus E}.
\end{equation}
In an open Riemann surface $X$, $h(E)$ is the closure of the union of all holes of $E$. 

The following notion is due to P.\ Manne, E.\ F.\ Wold, and N.\ {\O}vrelid 
\cite[Definition 2.1]{ManneWoldOvrelid2011}; see also 
\cite{MagnussonWold2016} and \cite[Sect.\ 2.1]{Chenoweth2018}.

%
%

\begin{definition}\label{def:BEH2}
A closed subset $E$ of a Stein manifold $X$ is said to have {\em bounded exhaustion hulls property} 
if there is a normal exhaustion $K_1\subset K_2\subset \cdots \subset  \bigcup_i K_i=X$
by compact $\Ocal(X)$-convex sets such that the set $h(E\cup K_i)$
is compact for every $i=1,2,\ldots$.
\end{definition}

For closed sets in $X=\mathbb{C}$, this coincides with Definition \ref{def:BEH}.
Just as in this special case, a closed set $E$ in an open Riemann surface $X$ is an Arakelian set
(i.e., the complement of $E$ in the one point compactification of $X$ is connected and locally connected)
if and only if $\wh E=E$ and $E$ has bounded exhaustion hulls property.

Given a compact set $K$ in an open Riemann surface $X$, 
we denote by $\Cscr^{0,1}(K)$ the space of $(0,1)$-forms
with continuous coefficients on $K$, endowed with the supremum norm. 
Choose a nowhere vanishing holomorphic $1$-form 
$\theta$ on $X$ (see \cite[Theorem 5.3.1 (c)]{Forstneric2017E}). Then, the map
$\Cscr(K) \ni g \mapsto g \bar\theta \in \Cscr^{0,1}(K)$ is an isomorphism of Banach space.

We now introduce the key analytic condition used in our proof of Arakelian's theorem.
We denote by $\Cscr_b(X)$ the Banach space of bounded continuous functions
on $X$ endowed with the supremum norm.

%
%
\begin{definition}\label{def:CondBH}
An open Riemann surface $X$ satisfies {\em Condition BH} 
if for every compact set $K$ in $X$ there is a bounded linear operator 
$T_K:\Cscr^{0,1}(K)\to \Cscr_b(X)$ satisfying the following two conditions.
\begin{enumerate}
\item For every $g\in \Cscr^{0,1}(K)$ we have $\dibar\, T_K(g)=g$ in the sense of distributions,
where we take $g=0$ on $X\setminus K$ (so the function $T_K(g)$ is holomorphic there).
\item $T_K$ is a {\em holomorphic operator} in the following sense:
if $g(\cdotp,w)\in \Cscr^{0,1}(K)$ is a family depending 
holomorphically on a parameter $w\in W\subset \mathbb{C}^n$, then 
$T_K(g(\cdotp,w)) \in \Cscr_b(X)$ also depends holomorphically on $w$.
\end{enumerate}
\end{definition}

The acronym BH  in the above definition stands for {\em bounded holomorphic}, the two 
key properties of operators $T_K$. On any domain $X\subset \mathbb{C}$ the Cauchy-Green operator 
$T_K$ \eqref{eq:TK} satisfies these and even stronger conditions as discussed
in the Introduction. We shall consider further examples on more general
Riemann surface below.

%
%
\begin{theorem}
\label{th:Arakelian2}
Let $X$ be an open Riemann surface satisfying Condition BH
(see Definition \ref{def:CondBH}), and let $E\subset X$
be a closed Arakelian set. If $Y$ is a compact complex homogeneous manifold, 
then every continuous map $X\to Y$ which is holomorphic in $\mathring E$ 
can be approximated uniformly on $E$ by holomorphic maps $X\to Y$.
\end{theorem}

\begin{proof}
The assumption that $E$ is Arakelian means that $E$ is $\Ocal(X)$-convex 
and has bounded exhaustion hulls property 
(see Definition \ref{def:BEH2}). For any compact set $K\subset X$, the set 
$h(E\cup K)$ (see \eqref{eq:hE}) is compact, and  the hull 
$\wh{E\cup K}=E\cup K\cup  h(E\cup K)$ is closed and 
$\Ocal(X)$-convex (see \cite[Lemma 3]{Chenoweth2018}).
Hence, we can find a sequence of compact, smoothly bounded, $\Ocal(X)$-convex subsets 
$\Delta_1\subset \Delta_2\subset \cdots \subset \bigcup_{i=1}^\infty \Delta_i=X$ such that, 
letting $H_i = h(E\cup \Delta_i)$ (see \eqref{eq:hE}),  
we have $\Delta_i\cup H_i \subset \mathring \Delta_{i+1}$ for all $i\in\N$. 
As in \eqref{eq:Ei} we set
\begin{equation}\label{eq:Ei2}
	E_0=E,\qquad E_i=E\cup\Delta_i\cup H_i = \wh{E\cup\Delta_i}, \quad i=1,2,\ldots.
\end{equation}
We now apply the inductive construction in the proof of Theorem \ref{th:Arakelian1}. 
Condition BH ensures that Lemma \ref{lem:linear-splitting} remains valid, with the same proof.
The same holds for Lemma \ref{lem:splitting} whose proof 
uses Lemma \ref{lem:linear-splitting} and the assumption that the operator $T_K$ is holomorphic
in the sense of condition (2) in Definition \ref{def:CondBH}. 
It remains to follow the proof of Theorem \ref{th:Arakelian1}. 
\end{proof}

%
%
Natural operators that may satisfy Condition BH 
are given by Cauchy kernels. H.\ Behnke and K.\ Stein \cite[Theorem 3]{BehnkeStein1949}
constructed Cauchy kernels, also called \emph{elementary differentials}, on 
any relatively compact domain in an open Riemann surface $X$.  
(See also H.\ Behnke and F.\ Sommer \cite[p.\ 584]{BehnkeSommer1962}.)
Globally defined Cauchy kernels were constructed 
by S.\ Scheinberg \cite{Scheinberg1978} and P.\ M.\ Gauthier \cite{Gauthier1979} in 1978--79. 
A Cauchy kernel is a meromorphic 1-form $\omega$ on 
$X\times X$ which is holomorphic off the diagonal and has the form
\[ 
	\omega(z,\zeta) = \left(\frac{1}{\zeta - z} + h(z,\zeta)\right) d\zeta
\] 
in any pair of local holomorphic coordinates on $X$, where $h$ is a holomorphic function. 
In particular, $\omega$ has simple poles with 
residues one along the diagonal of $X\times X$ and no other poles.  
For any relatively compact domain $\Omega\subset X$ with piecewise $\Cscr^1$ boundary 
and function $f\in\Cscr^1(\overline\Omega)$ one has the generalised Cauchy-Green formula:
\[ 
	f(z) = \frac{1}{2\pi \imath}\int_{\partial\Omega} f(\zeta)\, \omega(z,\zeta) 
	- \frac{1}{2\pi \imath} \int_{\Omega}\overline\partial f(\zeta)\wedge\omega(z,\zeta),
	\quad z\in \Omega.
\] 
Given a compact subset $K\subset X$, the linear operator $T_K:\Cscr^{0,1}(K)\to \Cscr(X)$, 
\[
	T_K(g)(z) = - \frac{1}{2\pi \imath} \int_{\zeta \in K} 
	g(\zeta) \wedge\omega(z,\zeta),\quad g\in \Cscr^{0,1}(K),\ z\in X
\]
satisfies condition (2) in Definition \ref{def:CondBH}, 
and it also satisfies condition (1) on any relatively compact domain in $X$ 
containing $K$. When $X=\mathbb{C}$ and $\omega(z,\zeta) = \frac{d\zeta}{\zeta - z}$, 
$T_K$ is the classical Cauchy-Green operator \eqref{eq:TK}.
This gives the following corollary to Theorem \ref{th:Arakelian2}.

\begin{corollary}\label{cor:Arakelian2}
If $X$ is a relatively compact domain in an open Riemann surface $R$ 
then Arakelian's theorem (the conclusion of Theorem \ref{th:Arakelian2}) holds
for maps from $X$ to an arbitrary compact complex homogeneous manifold.
\end{corollary}

The corollary applies in particular to any bordered Riemann surface and,
more generally, to any open Riemann surface $X$ of finite topological type
(i.e., of finite genus and finitely many ends) such that at least one end of $X$ is not a puncture. 
By a theorem of  E.\ L.\ Stout \cite{Stout1965}, such Riemann surface $X$ is biholomorphic 
to a relatively compact domain in another open Riemann surface.

%
%

\section{A remark on elliptic manifolds}\label{sec:remark}

We recall the following notion due to M.\ Gromov \cite{Gromov1989}; 
see also \cite[Definition 5.6.13]{Forstneric2017E}.

%
%
\begin{definition}\label{def:elliptic}
A complex manifold $Y$ is {\em elliptic} if there is a triple  $(\Ecal,\pi,s)$, 
where $\pi \colon \Ecal\to Y$ is a holomorphic vector bundle (the {\em spray bundle})
and $s\colon \Ecal\to Y$ is a holomorphic map (the {\em spray map}), 
such that for every point $y\in Y$ we have that $s(0_y)=y$ and the differential 
$ds_{0_y}\colon T_{0_y}\Ecal \to T _y Y$ maps the vertical subspace $\Ecal_y=\pi^{-1}(y)$ of 
the tangent space $T_{0_y}\Ecal$ surjectively onto $T_y Y$.
(Here, $0_y$ denotes the zero element of $\Ecal_y$ and we identified $T_{0_y} \Ecal_y$ 
with $\Ecal_y$.) Such $(\Ecal,\pi,s)$ is called a {\em dominating (holomorphic) spray} on $Y$.

The manifold $Y$ is {\em special elliptic} if it admits a dominating spray defined
on a trivial bundle $\Ecal=Y\times \mathbb{C}^n$ for some $n\ge \dim Y$.  
\end{definition}

Given a dominating spray $s\colon \Ecal\to Y$, we denote by 
\begin{equation}\label{eq:VD}
	Ds(y) = ds_{0_y}|_{ \Ecal_y} : \Ecal_y\lra T_y Y,\qquad y\in Y,
\end{equation}
the {\em vertical derivative} of $s$ at the point $0_y=y\times \{0\}^n$; this is the restriction
of the differential $ds_{0_y}:T_{0_y}\Ecal\to T_yY$ to the vertical subspace 
$T_{0_y} \Ecal_y \cong \Ecal_y$ of $T_{0_y}\Ecal$. 

Gromov proved in \cite{Gromov1989} that every elliptic manifold is an Oka manifold;
see \cite{ForstnericPrezelj2002} and \cite[Theorem 6.2.2]{Forstneric2017E} for the details.
Further information on elliptic manifolds can be found in \cite[Chaps.\ 5--7]{Forstneric2017E}.)
There is no known example of a nonelliptic Oka manifold.
For a Stein manifold, the properties of being Oka, elliptic, and special elliptic coincide 
(see \cite[3.2.A]{Gromov1989} or \cite[Proposition 5.6.15]{Forstneric2017E}).
Note that every complex homogeneous manifold $Y$ is special elliptic since
it admits a dominating spray of the form \eqref{eq:spray}.
However, there exist many nonhomogeneous special elliptic manifolds; for example, 
complements $\mathbb{C}^n\setminus A$ of affine algebraic subvarieties 
$A\subset \mathbb{C}^n$ of dimension $\le n-2$ (see \cite[Proposition 6.4.1]{Forstneric2017E}).

Our proof of Theorem \ref{th:Arakelian1} given in Sect.\ \ref{sec:proof1} 
only depends on the existence of a dominating spray on $Y$ defined on a trivial bundle,
so it applies to any compact special elliptic manifold. 
We note here that every such manifold is actually complex homogeneous.

%
%
\begin{proposition}\label{prop:SEhomogeneous}
Every compact  special elliptic manifold is complex homogeneous.
\end{proposition}


\begin{proof}
Let $s\colon \Ecal=Y\times \mathbb{C}^n\to Y$ be a dominating spray on a compact manifold $Y$. 
Recall that $Ds$ denotes the vertical derivative 
\eqref{eq:VD} of $s$ at the zero section $Y\times \{0\}$ of $\Ecal$. 
Given a holomorphic section $\xi\colon Y\to \Ecal$, the map 
\[
	Y\ni y\ \longmapsto\ V(y) := Ds(y)(\xi(y)) \in T_yY
\]
is then a holomorphic vector field on $Y$. (The point $\xi(y)$ on the right hand side 
of the formula is considered as an element of $T_{0_y} \Ecal_y$ which we can identify with 
$\Ecal_y\cong \mathbb{C}^n$.)

Applying this argument to basis sections $\xi_1,\ldots,\xi_n$  of the trivial bundle 
$\Ecal=Y\times \mathbb{C}^n$ yields 
holomorphic vector fields $V_1,\ldots,V_n$ on $Y$, and the domination property of $s$
implies that their values at any point $y\in Y$ span the tangent space $T_y Y$. 
Thus, the manifold $Y$ is {\em flexible} in the sense of Arzhantsev et al.\ 
\cite{ArzhantsevFlennerKalimanKutzschebauchZaidenberg2013}.
Since $Y$ is compact, these vector fields are complete, and hence their flows 
are complex $1$-parameter subgroups of the holomorphic automorphism group $\Aut(Y)$.
The spanning property (flexibility) easily implies that $\Aut(Y)$ acts transitively on $Y$. 
Since the holomorphic automorphism group of a compact complex manifold is a finite
dimensional complex Lie group, it follows that $Y$ is a homogeneous space of the 
complex Lie group $\Aut(Y)$. 
\end{proof}

There exist many nonhomogeneous compact Oka manifolds, for instance,
blowups of certain compact algebraic manifolds such as projective spaces, Grassmanians,
etc.; see \cite[Propositions 6.4.5 and 6.4.6]{Forstneric2017E} 
and the papers \cite{KalimanKutzschebauchTruong2018,LarussonTruong2017}.
It is not known whether every Oka manifold is elliptic, but Proposition \ref{prop:SEhomogeneous} 
tells us that there are (compact) Oka manifolds which are not elliptic, 
or there are elliptic manifolds which are not special elliptic.


\subsection*{Acknowledgements}
The author is supported by the research program P1-0291 and the grants J1-7256
and J1-9104 from ARRS, Republic of Slovenia.

I wish to thank Oliver Dragi{\v c}evi{\'c} for providing
a suitable reference concerning the properties of the Cauchy-Green operator, 
Frank Kutzschebauch and Finnur L{\'a}russon for their helpful remarks on elliptic and complex 
homogeneous manifolds, and Evgeny Poletsky for a useful discussion concerning Theorem \ref{th:Poletsky3.1} and Remark \ref{rem:Poletsky}. 
In particular, I thank F.\ Kutzschebauch for the argument that every flexible 
compact complex manifold is complex homogeneous. 
I also thank an anonymous referee for having pointed out a gap in the original
manuscript, and for the remarks which hopefully led to improved presentation.



\begin{thebibliography}{10}

\bibitem{Abraham1963}
R.~Abraham.
\newblock Transversality in manifolds of mappings.
\newblock {\em Bull. Amer. Math. Soc.}, 69:470--474, 1963.

\bibitem{Ahlfors2006}
L.~V. Ahlfors.
\newblock {\em Lectures on quasiconformal mappings}, volume~38 of {\em
  University Lecture Series}.
\newblock American Mathematical Society, Providence, RI, second edition, 2006.
\newblock With supplemental chapters by C. J. Earle, I. Kra, M. Shishikura and
  J. H. Hubbard.

\bibitem{Arakelian1964}
N.~U. Arakelian.
\newblock Uniform approximation on closed sets by entire functions.
\newblock {\em Izv. Akad. Nauk SSSR Ser. Mat.}, 28:1187--1206, 1964.

\bibitem{Arakelian1968}
N.~U. Arakelian.
\newblock Uniform and tangential approximations by analytic functions.
\newblock {\em Izv. Akad. Nauk Armjan. SSR Ser. Mat.}, 3(4-5):273--286, 1968.

\bibitem{Arakelian1971}
N.~U. Arakelian.
\newblock Approximation complexe et propri\'et\'es des fonctions analytiques.
\newblock In {\em Actes du {C}ongr\`es {I}nternational des {M}ath\'ematiciens
  ({N}ice, 1970), {T}ome 2}, pages 595--600. Gauthier-Villars, Paris, 1971.

\bibitem{ArzhantsevFlennerKalimanKutzschebauchZaidenberg2013}
I.~Arzhantsev, H.~Flenner, S.~Kaliman, F.~Kutzschebauch, and M.~Zaidenberg.
\newblock Infinite transitivity on affine varieties.
\newblock In {\em Birational geometry, rational curves, and arithmetic}, pages
  1--13. Springer, New York, 2013.

\bibitem{AstalaIwaniecMartin2009}
K.~Astala, T.~Iwaniec, and G.~Martin.
\newblock {\em Elliptic partial differential equations and quasiconformal
  mappings in the plane}, volume~48 of {\em Princeton Mathematical Series}.
\newblock Princeton University Press, Princeton, NJ, 2009.

\bibitem{BehnkeSommer1962}
H.~Behnke and F.~Sommer.
\newblock {\em Theorie der analytischen {F}unktionen einer komplexen
  {V}er\"anderlichen}.
\newblock Zweite ver\"anderte Auflage. Die Grundlehren der mathematischen
  Wissenschaften, Bd. 77. Springer-Verlag, Berlin-G\"ottingen-Heidelberg, 1962.

\bibitem{BehnkeStein1949}
H.~Behnke and K.~Stein.
\newblock Entwicklung analytischer {F}unktionen auf {R}iemannschen {F}l\"achen.
\newblock {\em Math. Ann.}, 120:430--461, 1949.

\bibitem{Bishop1958DMJ}
E.~Bishop.
\newblock The structure of certain measures.
\newblock {\em Duke Math. J.}, 25:283--289, 1958.

\bibitem{Bishop1958PJM}
E.~Bishop.
\newblock Subalgebras of functions on a {R}iemann surface.
\newblock {\em Pacific J. Math.}, 8:29--50, 1958.

\bibitem{BoivinGauthier2001}
A.~Boivin and P.~M. Gauthier.
\newblock Holomorphic and harmonic approximation on {R}iemann surfaces.
\newblock In {\em Approximation, complex analysis, and potential theory
  ({M}ontreal, {QC}, 2000)}, volume~37 of {\em NATO Sci. Ser. II Math. Phys.
  Chem.}, pages 107--128. Kluwer Acad. Publ., Dordrecht, 2001.

\bibitem{BoivinJiang2004}
A.~Boivin and B.~Jiang.
\newblock Uniform approximation by meromorphic functions on {R}iemann surfaces.
\newblock {\em J. Anal. Math.}, 93:199--214, 2004.

\bibitem{Chenoweth2018}
B.~{Chenoweth}.
\newblock {Carleman Approximation of Maps into Oka Manifolds}.
\newblock {\em ArXiv e-prints}, Apr. 2018.

\bibitem{Coltoiu1990}
M.~Col{\c{t}}oiu.
\newblock Complete locally pluripolar sets.
\newblock {\em J. Reine Angew. Math.}, 412:108--112, 1990.

\bibitem{Demailly1990}
J.-P. Demailly.
\newblock Cohomology of {$q$}-convex spaces in top degrees.
\newblock {\em Math. Z.}, 204(2):283--295, 1990.

\bibitem{Dietmair1993}
T.~Dietmair.
\newblock The {M}ergelyan-{B}ishop theorem in the case of mappings into
  homogeneous manifolds. {II}. {P}roof of the theorem and application.
\newblock {\em Math. Ann.}, 295(3):505--512, 1993.

\bibitem{DocquierGrauert1960}
F.~Docquier and H.~Grauert.
\newblock Levisches {P}roblem und {R}ungescher {S}atz f\"ur {T}eilgebiete
  {S}teinscher {M}annigfaltigkeiten.
\newblock {\em Math. Ann.}, 140:94--123, 1960.

\bibitem{DrinovecForstneric2008FM}
B.~Drinovec~Drnov{\v{s}}ek and F.~Forstneri{\v{c}}.
\newblock Approximation of holomorphic mappings on strongly pseudoconvex
  domains.
\newblock {\em Forum Math.}, 20(5):817--840, 2008.

\bibitem{FFW2018}
J.~E. {Forn{\ae}ss}, F.~{Forstneri{\v c}}, and E.~{Wold}.
\newblock {Holomorphic approximation: the legacy of Weierstrass, Runge,
  Oka-Weil, and Mergelyan}.
\newblock {\em ArXiv e-prints}, Feb. 2018.
\newblock To appear in the volume {\em Advancements in Complex Analysis} by
  Springer-Verlag.

\bibitem{Forstneric2004AIF}
F.~Forstneri{\v{c}}.
\newblock Holomorphic submersions from {S}tein manifolds.
\newblock {\em Ann. Inst. Fourier (Grenoble)}, 54(6):1913--1942 (2005), 2004.

\bibitem{Forstneric2007AJM}
F.~Forstneri{\v{c}}.
\newblock Manifolds of holomorphic mappings from strongly pseudoconvex domains.
\newblock {\em Asian J. Math.}, 11(1):113--126, 2007.

\bibitem{Forstneric2013KAWA}
F.~Forstneri{\v{c}}.
\newblock Oka manifolds: from {O}ka to {S}tein and back.
\newblock {\em Ann. Fac. Sci. Toulouse Math. (6)}, 22(4):747--809, 2013.
\newblock With an appendix by Finnur L{\'a}russon.

\bibitem{ForstnericLarusson2011}
F.~Forstneri{\v{c}} and F.~L{\'a}russon.
\newblock Survey of {O}ka theory.
\newblock {\em New York J. Math.}, 17A:11--38, 2011.

\bibitem{ForstnericLowOvrelid2001}
F.~Forstneri{\v{c}}, E.~L{\o}w, and N.~{\O}vrelid.
\newblock Solving the {$d$}- and {$\overline\partial$}-equations in thin tubes
  and applications to mappings.
\newblock {\em Michigan Math. J.}, 49(2):369--416, 2001.

\bibitem{ForstnericPrezelj2002}
F.~Forstneri{\v{c}} and J.~Prezelj.
\newblock Oka's principle for holomorphic submersions with sprays.
\newblock {\em Math. Ann.}, 322(4):633--666, 2002.

\bibitem{Forstneric2017E}
F.~Forstneri\v{c}.
\newblock {\em {Stein manifolds and holomorphic mappings. The homotopy
  principle in complex analysis (2nd edn).}}, volume~56 of {\em Ergebnisse der
  Mathematik und ihrer Grenzgebiete. 3. Folge}.
\newblock Berlin: Springer, 2017.

\bibitem{Gamelin1984}
T.~W. Gamelin.
\newblock {\em Uniform algebras, 2nd ed.}
\newblock Chelsea Press, New York, 1984.

\bibitem{GamelinKhavinson1989}
T.~W. Gamelin and D.~Khavinson.
\newblock The isoperimetric inequality and rational approximation.
\newblock {\em Amer. Math. Monthly}, 96(1):18--30, 1989.

\bibitem{Gauthier1979}
P.~M. Gauthier.
\newblock Meromorphic uniform approximation on closed subsets of open {R}iemann
  surfaces.
\newblock In {\em Approximation theory and functional analysis ({P}roc.
  {I}nternat. {S}ympos. {A}pproximation {T}heory, {U}niv. {E}stadual de
  {C}ampinas, {C}ampinas, 1977)}, volume~35 of {\em North-Holland Math. Stud.},
  pages 139--158. North-Holland, Amsterdam-New York, 1979.

\bibitem{Gauthier1980}
P.~M. Gauthier.
\newblock Analytic approximation on closed subsets of open riemann surfaces.
\newblock In {\em Constructive Function Theory '77 (Proc. Conf. Blagoevgrad)},
  pages 139--158. Bulgarian Acad. Sci., Sofia, 1980.

\bibitem{GauthierHengartner1975}
P.~M. Gauthier and W.~Hengartner.
\newblock Uniform approximation on closed sets by functions analytic on a
  {R}iemann surface.
\newblock pages 63--69, 1975.

\bibitem{Grauert1958MA}
H.~Grauert.
\newblock Analytische {F}aserungen \"uber holomorph-vollst{\"a}ndigen
  {R}{\"a}umen.
\newblock {\em Math. Ann.}, 135:263--273, 1958.

\bibitem{Gromov1989}
M.~Gromov.
\newblock Oka's principle for holomorphic sections of elliptic bundles.
\newblock {\em J. Amer. Math. Soc.}, 2(4):851--897, 1989.

\bibitem{Henkin1969}
G.~M. Henkin.
\newblock Integral representation of functions which are holomorphic in
  strictly pseudoconvex regions, and some applications.
\newblock {\em Mat. Sb. (N.S.)}, 78 (120):611--632, 1969.

\bibitem{HenkinLeiterer1984}
G.~M. Henkin and J.~Leiterer.
\newblock {\em Theory of functions on complex manifolds}, volume~60 of {\em
  Mathematische Lehrb\"ucher und Monographien, II. Abteilung: Mathematische
  Monographien [Mathematical Textbooks and Monographs, Part II: Mathematical
  Monographs]}.
\newblock Akademie-Verlag, Berlin, 1984.

\bibitem{KalimanKutzschebauchTruong2018}
S.~Kaliman, F.~Kutzschebauch, and T.~T. Truong.
\newblock On subelliptic manifolds.
\newblock {\em Israel J. Math.}, 228(1):229--247, 2018.

\bibitem{LarussonTruong2017}
F.~L\'arusson and T.~T. Truong.
\newblock Algebraic subellipticity and dominability of blow-ups of affine
  spaces.
\newblock {\em Doc. Math.}, 22:151--163, 2017.

\bibitem{LiebMichel2002}
I.~Lieb and J.~Michel.
\newblock {\em The {C}auchy-{R}iemann complex. Integral formulae and Neumann
  problem}.
\newblock Aspects of Mathematics, E34. Friedr. Vieweg \& Sohn, Braunschweig,
  2002.

\bibitem{MagnussonWold2016}
B.~S. Magnusson and E.~F.~s. Wold.
\newblock A characterization of totally real {C}arleman sets and an application
  to products of stratified totally real sets.
\newblock {\em Math. Scand.}, 118(2):285--290, 2016.

\bibitem{Manne1993}
P.~E. Manne.
\newblock Carleman approximation on totally real submanifolds of a complex
  manifold.
\newblock In {\em Several complex variables ({S}tockholm, 1987/1988)},
  volume~38 of {\em Math. Notes}, pages 519--528. Princeton Univ. Press,
  Princeton, NJ, 1993.

\bibitem{ManneWoldOvrelid2011}
P.~E. Manne, E.~F. Wold, and N.~{\O}vrelid.
\newblock Holomorphic convexity and {C}arleman approximation by entire
  functions on {S}tein manifolds.
\newblock {\em Math. Ann.}, 351(3):571--585, 2011.

\bibitem{Mergelyan1951}
S.~N. Mergelyan.
\newblock On the representation of functions by series of polynomials on closed
  sets.
\newblock {\em Doklady Akad. Nauk SSSR (N.S.)}, 78:405--408, 1951.

\bibitem{Poletsky2013}
E.~A. Poletsky.
\newblock Stein neighborhoods of graphs of holomorphic mappings.
\newblock {\em J. Reine Angew. Math.}, 684:187--198, 2013.

\bibitem{Arellano1969}
E.~Ram{\'\i}rez~de Arellano.
\newblock Ein {D}ivisionsproblem und {R}andintegraldarstellungen in der
  komplexen {A}nalysis.
\newblock {\em Math. Ann.}, 184:172--187, 1969/1970.

\bibitem{Range1986}
R.~M. Range.
\newblock {\em Holomorphic functions and integral representations in several
  complex variables}, volume 108 of {\em Graduate Texts in Mathematics}.
\newblock Springer-Verlag, New York, 1986.

\bibitem{RosayRudin1989}
J.-P. Rosay and W.~Rudin.
\newblock Arakelian's approximation theorem.
\newblock {\em Amer. Math. Monthly}, 96(5):432--434, 1989.

\bibitem{Runge1885}
C.~Runge.
\newblock Zur {T}heorie der {E}indeutigen {A}nalytischen {F}unctionen.
\newblock {\em Acta Math.}, 6(1):229--244, 1885.

\bibitem{Scheinberg1978}
S.~Scheinberg.
\newblock Uniform approximation by functions analytic on a {R}iemann surface.
\newblock {\em Ann. of Math. (2)}, 108(2):257--298, 1978.

\bibitem{Siu1976}
Y.~T. Siu.
\newblock Every {S}tein subvariety admits a {S}tein neighborhood.
\newblock {\em Invent. Math.}, 38(1):89--100, 1976/77.

\bibitem{Stout1965}
E.~L. Stout.
\newblock  Bounded holomorphic functions on finite {R}iemann surfaces. 
\newblock {\em Trans. Amer. Math. Soc.}, 120:255--285, 1965. 
 
\bibitem{Verdera1986PAMS}
J.~Verdera.
\newblock On {$C^m$} rational approximation.
\newblock {\em Proc. Amer. Math. Soc.}, 97(4):621--625, 1986.

\bibitem{Vitushkin1966}
A.~G. Vitu{\v s}kin.
\newblock Conditions on a set which are necessary and sufficient in order that
  any continuous function, analytic at its interior points, admit uniform
  approximation by rational fractions.
\newblock {\em Dokl. Akad. Nauk SSSR}, 171:1255--1258, 1966.

\bibitem{Vitushkin1967}
A.~G. Vitu{\v s}kin.
\newblock Analytic capacity of sets in problems of approximation theory.
\newblock {\em Uspehi Mat. Nauk}, 22(6 (138)):141--199, 1967.

\bibitem{Winkelmann1998}
J.~Winkelmann.
\newblock A {M}ergelyan theorem for mappings to {${\bf C}^2\setminus {\bf
  R}^2$}.
\newblock {\em J. Geom. Anal.}, 8(2):335--340, 1998.

\end{thebibliography}


\vspace*{5mm}
\noindent Franc Forstneri\v c

\noindent Faculty of Mathematics and Physics, University of Ljubljana, Jadranska 19, SI--1000 Ljubljana, Slovenia

\noindent Institute of Mathematics, Physics and Mechanics, Jadranska 19, SI--1000 Ljubljana, Slovenia

\noindent e-mail: {\tt franc.forstneric@fmf.uni-lj.si}

\end{document}